\documentclass[12pt,a4paper]{interact}
\usepackage{latexsym}
\usepackage{amsmath,amsthm,amssymb,amsfonts}
\usepackage{caption}

\newcommand{\refe}[1]{(\ref{#1})}

\allowdisplaybreaks

\numberwithin{equation}{section} 

\theoremstyle{plain}
\newtheorem{theorem}{Theorem}[section]

\theoremstyle{definition}

\theoremstyle{remark}
\newtheorem{remark}{Remark}

\begin{document}


\articletype{RESEARCH ARTICLE}

\title{Interlacing of zeros of Laguerre polynomials of equal and consecutive degree}
\author{\name{J. Arves\'u\textsuperscript{a}\thanks{The research of J. Arves\'u was funded by Agencia Estatal de Investigaci\'on of Spain, 
grant number PGC-2018-096504-B-C33}, K. Driver\textsuperscript{b},\thanks {The research of K. Driver was funded by the National Research Foundation of South Africa, Grant Number 115332}  and L.  Littlejohn\textsuperscript{c}}
\affil{\textsuperscript{a}Department of Mathematics, Universidad Carlos III de Madrid, 
Avda. de la Universidad, 30, 28911, Legan\'es, Spain; \textsuperscript{b}Department of Mathematics and Applied Mathematics, University of Cape Town, Cape Town 7708, South Africa; \textsuperscript{c}Department of Mathematics, Baylor University, One Bear Place 97328
Waco, TX 76798-7328, USA.}
}

\maketitle

\begin{abstract} We investigate  interlacing properties of zeros of Laguerre polynomials $ L_{n}^{(\alpha)}(x)$  and $ L_{n+1}^{(\alpha +k)}(x),$   $ \alpha > -1, $  where $ n \in \mathbb{N}$ and  $ k \in {\{ 1,2 }\}$.  We prove that,  in general, the zeros of these polynomials interlace partially and not fully. The sharp $t-$interval within which the zeros of two equal degree Laguerre polynomials $ L_n^{(\alpha)}(x)$  and  $ L_n^{(\alpha +t)}(x)$ are interlacing  for every $n \in \mathbb{N}$  and each  $ \alpha > -1$  is  $ 0 < t \leq 2,$ \cite{DrMu2}, and the sharp $t-$interval within which the zeros of two consecutive degree Laguerre polynomials $ L_n^{(\alpha)}(x)$  and  $ L_{n-1}^{(\alpha +t)}(x)$ are interlacing  for every $n \in \mathbb{N}$  and each  $ \alpha > -1$  is  $ 0 \leq t \leq 2,$ \cite{DrMu1}.  We derive conditions on $n \in \mathbb{N}$  and  $\alpha,$  $ \alpha > -1$   that determine the partial or full  interlacing  of  the  zeros of  $ L_n^{(\alpha)}(x)$  and the zeros of  $ L_n^{(\alpha + 2 + k)}(x),$  $ k \in {\{ 1,2 }\}$.  We also prove that partial interlacing holds  between the  zeros of  $ L_n^{(\alpha)}(x)$  and  $ L_{n-1}^{(\alpha + 2 +k )}(x)$ when $ k \in {\{ 1,2 }\},$  $n \in \mathbb{N}$  and  $ \alpha > -1$.    Numerical illustrations of interlacing and its breakdown are provided.  
\end{abstract}

\begin{keywords}
Laguerre polynomials; Zeros; Interlacing; Three-term recurrence relation
\end{keywords}

\section{Introduction} 

The zeros of two  polynomials of consecutive degree  in any orthogonal sequence $\{p_{n}(x)\} _{n=0}^\infty,$ where $p_n(x) $ has exact degree $n,$  satisfy the interlacing property that exactly one zero of $p_{n-1}(x)$ lies strictly between each pair of successive zeros of $p_n(x)$  for each $n \in \mathbb{N}, n \geq 2$.  This classical result goes back almost 150 years to Chebyshev, Markov and Stieltjes and one of its first applications was to Gauss quadrature, \cite [Ch.3]{Sze}.  Subsequently, Lubinsky \cite{Lub} generalized a quadrature formula proved by Simon \cite{Sim} for orthogonal polynomials $P$ of degree $ \leq n-2$ by using the interlacing properties of zeros of two real polynomials $R$ and $S$ of degree $n-1$ and $n$ respectively. He observed that other analytical methods can be used to prove the quadrature results obtained but that the proofs are substantially simplified when interlacing properties of zeros are used. In other applications,  interlacing properties of the zeros of Jacobi polynomials were  used to develop new methods for approximating the finite Hilbert transform \cite{Mas} while the interlacing property of zeros gives rise to stability tests for linear difference forms \cite{Loc}.  

\medskip

Richard Askey \cite{Ask} considered the  question of whether the zeros of  the two equal degree Jacobi  polynomials  $P_n^{(\alpha,\beta)}(x)$  and  $P_n^{(\alpha,\beta +1)}(x)$ have interlacing zeros in the sense that each open interval with endpoints at a pair of consecutive zeros of  $P_n^{(\alpha,\beta)}(x)$  contains exactly one zero of  $P_n^{(\alpha,\beta +1)}(x)$ and conversely. He proved that the zeros  of  $P_n^{(\alpha,\beta)}(x)$  and  $P_n^{(\alpha,\beta +1)}(x)$ are  interlacing for each $n \in \mathbb{N},$  $ \alpha > -1, \beta > -1,$   and  he  conjectured, based on graphical evidence, that  the zeros of  the Jacobi polynomials $P_n^{(\alpha,\beta)}(x)$ and  $P_n^{(\alpha,\beta +2)}(x)$   are  interlacing for each $n \in \mathbb{N}$, $ \alpha > -1, \beta > -1$.   Following Askey's result  and his conjecture in \cite{Ask}, a number of interlacing results for zeros of  parameter-dependent orthogonal polynomials of the same degree, and of different degree, corresponding to  different parameters, can be found in  \cite {DIR}, \cite {DrJo}, \cite{DrJoMb}, \cite{DrMu1}, \cite{DrMu2}.

\section {The zeros of ${L_{n}^{(\alpha)}(x)}$ and  ${L_{n+1}^{(\alpha +t)}(x)}$, $ t \in {\{ 1,2 }\},$  $n \in \mathbb{N},$  $ \alpha > -1$}

For each $ \alpha > -1,$ the sequence of Laguerre polynomials $\{L_{n}^{(\alpha)}(x)\} _{n=0}^\infty $ is  orthogonal on the positive real line with respect to the weight function $e^{-x} x^{\alpha}$.   For each $n \in \mathbb{N}, n \geq 1, \alpha > -1,$ the polynomial $ L_n^{(\alpha)}(x)$  has $n$ real, simple, positive zeros and the zeros of $ L_n^{(\alpha)}(x)$  and $ L_{n-1}^{(\alpha)}(x)$  are interlacing.   It was  proved \cite{DrMu1} that the zeros of $ L_{n}^{(\alpha)}(x)$  and $ L_{n-k}^{(\alpha +t)}(x),$   $ n \in \mathbb{N},  k \in {\{ 1,2,\dots, n-1}\},$ are interlacing for $ 0  \leq t \leq 2k,$ apart from common zeros, and  in the Stieltjes sense \cite{DrJo1}  for $k \geq 2,$  namely,  that each distinct zero of $ L_{n-k}^{(\alpha +t)}(x)$   lies strictly between a  pair of consecutive zeros of $ L_{n}^{(\alpha)}(x),$  excluding common zeros of $ L_{n-k}^{(\alpha +t)}(x)$  and   $ L_{n}^{(\alpha)}(x),$  when $ k \in {\{ 2,\dots, n-1}\}$. Further,  the $t-$ interval  $ 0 \leq t \leq 2k$  is sharp in order for interlacing to hold for every $n \in \mathbb{N},  k \in {\{ 1,2,\dots, n-1}\},$  excluding situations in which $ L_{n-k}^{(\alpha +t)}(x)$  and   $ L_{n}^{(\alpha)}(x)$  have any common zeros. The sharp interlacing $t-$ interval  for interlacing of zeros of equal degree Laguerre polynomials $ L_n^{(\alpha)}(x)$  and  $ L_n^{(\alpha +t)}(x)$  is  $ 0 <  t \leq 2$, \cite {DrMu2}.  At OPSFA 15,  Alan Sokal asked what is known about the interlacing of the zeros of the  Laguerre polynomials $L_n^{(\alpha)}(x)$  and  $ L_{n+k}^{(\alpha +t)}(x)$ where $t >0$ and $k \in \mathbb{N}.$   The simplest case to consider  is $t=1$ and $k=1$  and our first result considers interlacing of the zeros of  $L_n^{(\alpha)}(x)$  and  $ L_{n+1}^{(\alpha +1)}(x)$.

\begin{theorem}\label{Theorem 2.1.} Suppose $\{L_{n}^{(\alpha)}(x)\} _{n=0}^\infty $ is  a sequence of Laguerre polynomials with $ \alpha > -1$. Assume that $\alpha$ and $n \in \mathbb{N}$ are such that $ L_{n+1}^{(\alpha +1)}(x)$ and $ L_{n}^{(\alpha)}(x)$ have no common zeros.  If  $\{z_{i}\} _{i=1}^{n+1} $ are the zeros of  $ L_{n+1}^{(\alpha +1)}(x)$ in increasing order, then each of the $n+1$ intervals $(0, z_1), (z_1, z_2), \dots, (z_n, z_{n+1})$ contains either the point $n +1$ or exactly one (simple) zero of $L_n^{(\alpha)}(x)$ but not both.  The zeros of $L_n^{(\alpha)}(x)$  and  $ L_{n+1}^{(\alpha +1)}(x),$ $\alpha > -1$  are  interlacing  if and only if the smallest  positive zero $z_1$  of  $ L_{n+1}^{(\alpha +1)}(x)$ satisfies  $z_1 > n +1$. 
\end{theorem}

\begin{proof}  Let

\medskip

  $  0 < x_1 < x_2 < \dots < x_n $ denote  the zeros of  $L_n^{(\alpha)}(x);$  

\medskip

   $ 0 < y_1 < y_2 < \dots < y_n < y_{n+1}$  denote the zeros of $ L_{n+1}^{(\alpha)}(x);$  

\medskip

   $ 0 < z_1 < z_2 < \dots < z_n < z_{n+1}$ denote the zeros of $ L_{n +1}^{(\alpha +1)}(x)$.  

\medskip

From  \cite [(1)] {DrJo} with $n$ replaced by $n+1$:

\begin{equation}\label{(2.1)}  x L_{n+1}^{(\alpha +1)}(x) = (x - (n+1))  L_{n+1}^{(\alpha)}(x) + (\alpha + n +1)  L_n^{(\alpha)}(x).   \end{equation} 

It is known from \cite [Th. 2.3] {DrJo} that the zeros of  the two equal degree Laguerre polynomials $ L_{n+1}^{(\alpha)}(x)$  and  $ L_{n+1}^{(\alpha +1)}(x)$  are interlacing.  Further, Markov's monotonicity theorem shows that the zeros of $ L_{n}^{(\alpha +t)}(x)$ are increasing functions of the parameter $t$. Therefore, for each $n \in \mathbb{N},$

\begin{eqnarray}\label{(2.2)}   0 < y_1 < z_1 < y_2 < z_2 < \dots < y_{n+1} <  z_{n+1}. \end{eqnarray} 

Also, the zeros of  $ L_{n}^{(\alpha)}(x)$  and  $ L_{n+1}^{(\alpha)}(x)$  are interlacing for $\alpha > -1$  so that

\begin{eqnarray}\label{(2.3)}   0 < y_1 < x_1 < y_2 < x_2 < \dots < y_{n} <  x_{n} < y_{n+1}. \end{eqnarray}

\medskip

We are investigating  interlacing of the zeros of  $L_n^{(\alpha)}(x)$  and  $ L_{n+1}^{(\alpha +1)}(x)$   so we look for inequalities satisfied by  the zeros $x_i,$   $ i \in {\{ 1,2,\dots, n}\}$  of  $L_n^{(\alpha)}(x)$  and the zeros $z_i,$  $ i \in {\{ 1,2,\dots, n+1}\}$ of  $ L_{n+1}^{(\alpha +1)}(x)$. 

\medskip

Evaluating \refe{(2.1)}  at successive zeros $z_k$ and $z_{k+1} $   of  $ L_{n+1}^{(\alpha +1)}(x),$  we have for each $k \in {\{ 1,2,\dots, n}\}$

\begin{equation*}  (z_k - (n+1))  L_{n+1}^{(\alpha)}(z_k) =  - (\alpha + n +1)  L_n^{(\alpha)}(z_k)  \end{equation*} 

and 

\begin{equation*}  (z_{k+1} - (n+1))  L_{n+1}^{(\alpha)}(z_{k+1}) =  - (\alpha + n +1)  L_n^{(\alpha)}(z_{k+1}),   \end{equation*} 

and therefore, for each $k = {1,2,\dots, n},$

\begin{equation}\label{(2.4)}  [z_k - (n+1)]  [z_{k+1} - (n+1)] L_{n+1}^{(\alpha)}(z_k) L_{n+1}^{(\alpha)}(z_{k+1}) =   (\alpha + n +1)^2   L_n^{(\alpha)}(z_k)  L_n^{(\alpha)}(z_{k+1}).  \end{equation} 

The product  $ L_{n+1}^{(\alpha)}(z_k)  L_{n+1}^{(\alpha)}(z_{k+1}) < 0 $ for each $k = {1,2,\dots, n}$ since the zeros of  the two equal degree Laguerre polynomials $ L_{n+1}^{(\alpha)}(x)$  and  $ L_{n+1}^{(\alpha +1)}(x)$  are interlacing, \cite[Th. 2.3] {DrJo}. Also, $(\alpha + n +1)^2 >0,$  so  the right hand side  of \refe{(2.4)} is $<0$ unless $n +1 \in (z_k, z_{k+1})$ which can occur for at most one value of $k$.  

\medskip 

It follows that $L_n^{(\alpha)}(x)$ has at least $n-1$ sign changes between $z_1$ and $z_{n+1}$ and each of the $n$ intervals $(z_k, z_{k+1}),$  $k = {1,2,\dots, n},$   contains either (i) exactly one zero of  $L_n^{(\alpha)}(x)$ and not the point $n+1$ or (ii) the point $n+1$ and  no zero of  $L_n^{(\alpha)}(x)$.  If (i) holds, we have interlacing of the zeros of  $L_n^{(\alpha)}(x)$  and $L_{n+1}^{(\alpha +1)}(x)$.  If (ii) holds, $n-1$ zeros of  $L_n^{(\alpha)}(x),$ together with the point $x= n+1,$ interlace with the $n+1$ zeros of  $L_{n+1}^{(\alpha +1)}(x)$.  The  remaining zero of $L_n^{(\alpha)}(x)$   cannot lie in any of the intervals $(z_k, z_{k+1})$  that already contains  a  zero of $L_n^{(\alpha)}(x)$ nor can it lie in the interval $(z_k, z_{k+1})$ that contains the point $n+1,$ since that would contradict the number of sign changes in each interval $(z_k, z_{k+1}),$ $ k \in {\{ 1,2,\dots, n}\}$.  From \refe{(2.2)} and \refe{(2.3)},  $x_n < y_{n+1} <  z_{n+1}$  so the remaining  zero of  $ L_{n}^{(\alpha)}(x)$  must be less  than $z_1$.  With reference to equations \refe{(2.2)} and \refe{(2.3)}, we have, in case (ii), 

\begin{eqnarray*}\label{(2.5)}   0 <y_1 <  x_1 < z_1 < y_2 <  \dots   \end{eqnarray*}

This completes the proof of Theorem \ref{Theorem 2.1.}.
\end{proof}

\begin{remark} The "extra point", namely $n+1$, which "completes" the interlacing of the zeros  of  $ L_{n+1}^{(\alpha +1)}(x)$  and  $L_n^{(\alpha)}(x),$  $\alpha > -1,$ depends only on  $n\in \mathbb{N}$ and is independent of $\alpha$.
\end{remark}

Table 1 and Table 2 list, respectively, the zeros of $L_{n+1}^{(\alpha +1)}(x)$ and $L_{n}^{(\alpha)}(x)$ for $\alpha \in \{0,1,\dots,19\}$ and $n = 7$. The numerical evidence confirms the results proved in Theorem \ref{Theorem 2.1.}. The boxed points highlight the interval with endpoints at successive zeros of $L_{n+1}^{(\alpha +1)}(x)$  and  $L_{n}^{(\alpha)}(x)$ respectively that contain the point $x = n+1 =8$.

\begin{table}[!htbp]
\caption{The zeros of $ L_{n+1}^{(\alpha +1)}(x)$ for $\alpha \in\{0,1,\dots,19\}$ and $n = 7$.}
$$
\begin{array}{lllllllll}
L^{(\alpha+1)}_{n+1}(x) & z_1  & z_2  & z_3  & z_4  & z_5  & z_6  & z_7  & z_8 
\\
\hline\hline
\\
L^{1}_{8}(x) & 0.409  & 1.38  & 2.96  & \fbox{5.18}  & \fbox{8.16}  & 12.1  & 17.2  & 24.6 
   \\
L^{2}_{8}(x) & 0.699  & 1.90  & 3.68  & \fbox{6.10}  & \fbox{9.27}  & 13.4  & 18.7  & 26.3 
   \\
L^{3}_{8}(x) & 1.03  & 2.44  & 4.41  & \fbox{7.02}  & \fbox{10.4}  & 14.6  & 20.2  & 27.9 
   \\
L^{4}_{8}(x) & 1.39  & 3.00  & 5.16  & \fbox{7.94}  & \fbox{11.5}  & 15.9  & 21.6  & 29.5 
   \\
L^{5}_{8}(x) & 1.79  & 3.59  & \fbox{5.92}  & \fbox{8.87}  & 12.5  & 17.1  & 23.0  & 31.1 
   \\
L^{6}_{8}(x) & 2.21  & 4.19  & \fbox{6.69}  & \fbox{9.79}  & 13.6  & 18.4  & 24.4  & 32.7 
   \\
L^{7}_{8}(x) & 2.65  & 4.81  & \fbox{7.47}  & \fbox{10.7}  & 14.7  & 19.6  & 25.8  & 34.2 
   \\
L^{8}_{8}(x) & 3.11  & \fbox{5.44}  & \fbox{8.25}  & 11.7  & 15.8  & 20.8  & 27.2  & 35.8 
   \\
L^{9}_{8}(x) & 3.59  & \fbox{6.08}  & \fbox{9.04}  & 12.6  & 16.9  & 22.0  & 28.5  & 37.3 
   \\
L^{10}_{8}(x) & 4.08  & \fbox{6.73}  & \fbox{9.84}  & 13.5  & 17.9  & 23.2  & 29.9  & 38.8 
   \\
L^{11}_{8}(x) & 4.59  & \fbox{7.40}  & \fbox{10.6}  & 14.5  & 19.0  & 24.4  & 31.2  & 40.3 
   \\
L^{12}_{8}(x) & \fbox{5.11}  & \fbox{8.07}  & 11.5  & 15.4  & 20.1  & 25.6  & 32.5  & 41.8 
   \\
L^{13}_{8}(x) & \fbox{5.64}  & \fbox{8.75}  & 12.3  & 16.3  & 21.1  & 26.8  & 33.8  & 43.2 
   \\
L^{14}_{8}(x) & \fbox{6.19}  & \fbox{9.44}  & 13.1  & 17.3  & 22.2  & 28.0  & 35.1  & 44.7 
   \\
L^{15}_{8}(x) & \fbox{6.75}  & \fbox{10.1}  & 13.9  & 18.2  & 23.2  & 29.2  & 36.5  & 46.1 
   \\
L^{16}_{8}(x) & \fbox{7.31}  & \fbox{10.8}  & 14.7  & 19.2  & 24.3  & 30.4  & 37.7  & 47.5 
   \\
L^{17}_{8}(x) & \fbox{7.89}  & \fbox{11.6}  & 15.6  & 20.1  & 25.4  & 31.5  & 39.0  & 49.0 
   \\
L^{18}_{8}(x) & 8.47  & 12.3  & 16.4  & 21.1  & 26.4  & 32.7  & 40.3  & 50.4 
   \\
L^{19}_{8}(x) & 9.07  & 13.0  & 17.2  & 22.0  & 27.5  & 33.9  & 41.6  & 51.8 
   \\
L^{20}_{8}(x) & 9.67  & 13.7  & 18.1  & 23.0  & 28.5  & 35.0  & 42.9  & 53.2 
\end{array}
$$
\label{Laguerre-1}
\end{table}


\begin{table}[!htbp]
\caption{The zeros of $L_{n}^{(\alpha)}(x)$ for $\alpha \in\{0,1,\dots,19\}$ and $n = 7$.}
$$
\begin{array}{llllllll}
L^{(\alpha)}_{n}(x) & x_1  & x_2  & x_3  & x_4  & x_5  & x_6  & x_7 
\\
\hline\hline
\\
L^{0}_{7}(x) & 0.193  & 1.03  & 2.57  & \fbox{4.90}  & \fbox{8.18}  & 12.7  & 19.4  \\
L^{1}_{7}(x) & 0.461  & 1.56  & 3.35  & \fbox{5.92}  & \fbox{9.42}  & 14.2  & 21.1  \\
L^{2}_{7}(x) & 0.783  & 2.13  & 4.15  & \fbox{6.93}  & \fbox{10.6}  & 15.6  & 22.7  \\
L^{3}_{7}(x) & 1.15  & 2.73  & 4.96  & \fbox{7.94}  & \fbox{11.8}  & 17.0  & 24.4  \\
L^{4}_{7}(x) & 1.55  & 3.34  & \fbox{5.77}  & \fbox{8.94}  & 13.0  & 18.4  & 25.9  \\
L^{5}_{7}(x) & 1.98  & 3.98  & \fbox{6.60}  & \fbox{9.95}  & 14.2  & 19.8  & 27.5  \\
L^{6}_{7}(x) & 2.43  & 4.63  & \fbox{7.43}  & \fbox{11.0}  & 15.4  & 21.1  & 29.0  \\
L^{7}_{7}(x) & 2.91  & \fbox{5.30}  & \fbox{8.27}  & 12.0  & 16.6  & 22.4  & 30.6  \\
L^{8}_{7}(x) & 3.40  & \fbox{5.98}  & \fbox{9.11}  & 13.0  & 17.7  & 23.8  & 32.1  \\
L^{9}_{7}(x) & 3.92  & \fbox{6.67}  & \fbox{9.96}  & 14.0  & 18.9  & 25.1  & 33.5  \\
L^{10}_{7}(x) & 4.45  & \fbox{7.37}  & \fbox{10.8}  & 15.0  & 20.0  & 26.4  & 35.0  \\
L^{11}_{7}(x) & \fbox{4.99}  & \fbox{8.08}  & 11.7  & 16.0  & 21.2  & 27.7  & 36.5  \\
L^{12}_{7}(x) & \fbox{5.55}  & \fbox{8.79}  & 12.5  & 17.0  & 22.3  & 29.0  & 37.9  \\
L^{13}_{7}(x) & \fbox{6.12}  & \fbox{9.52}  & 13.4  & 18.0  & 23.4  & 30.2  & 39.3  \\
L^{14}_{7}(x) & \fbox{6.70}  & \fbox{10.3}  & 14.3  & 19.0  & 24.6  & 31.5  & 40.7  \\
L^{15}_{7}(x) & \fbox{7.29}  & \fbox{11.0}  & 15.1  & 20.0  & 25.7  & 32.8  & 42.1  \\
L^{16}_{7}(x) & \fbox{7.88}  & \fbox{11.7}  & 16.0  & 21.0  & 26.8  & 34.0  & 43.5  \\
L^{17}_{7}(x) & 8.49  & 12.5  & 16.9  & 22.0  & 28.0  & 35.3  & 44.9  \\
L^{18}_{7}(x) & 9.11  & 13.2  & 17.8  & 23.0  & 29.1  & 36.5  & 46.3  \\
L^{19}_{7}(x) & 9.73  & 14.0  & 18.7  & 24.0  & 30.2  & 37.7  & 47.7 
\end{array}
$$
\label{Laguerre-2}
\end{table}

\newpage
We now consider whether the zeros of  two Laguerre polynomials of consecutive degree $n$ and $n+1$ corresponding to the parameters $\alpha$ and $\alpha +2,$  namely the zeros of $L_n^{(\alpha)}(x)$  and  $ L_{n+1}^{(\alpha +2)}(x),$ $\alpha > -1$, are interlacing. 

\begin{theorem}\label{Theorem 2.2.} Suppose $\{L_{n}^{(\alpha)}(x)\} _{n=0}^\infty $ is  a sequence of Laguerre polynomials with $ \alpha > -1$. Assume that $\alpha$ and $n \in \mathbb{N}$ are such that $ L_{n+1}^{(\alpha +2)}(x)$ and $ L_{n}^{(\alpha)}(x)$ have no common zeros.   If  $\{w_{i}\} _{i=1}^{n+1} $ are the zeros of  $ L_{n+1}^{(\alpha +2)}(x)$   in increasing order, then each of the $n+1$ intervals $(0, w_1), (w_1, w_2), \dots, (w_n, w_{n+1})$ contains either the point $ n + 1 + \sqrt{( n +1)(n + \alpha +2)}$ or exactly one (simple) zero of $L_n^{(\alpha)}(x),$ but not both.  The zeros of $L_n^{(\alpha)}(x)$  and  $ L_{n+1}^{(\alpha +2)}(x),$ $\alpha > -1,$  are interlacing if and only if the point  $ n +1   + \sqrt{( n +1)(n + \alpha +2)}$  lies in the open  interval $(0, w_1)$.  This occurs when $\alpha >>n.$ 
\end{theorem}

\begin{proof} Let

\medskip

  $  0 < x_1 < x_2 < \dots < x_n $ denote  the zeros of  $L_n^{(\alpha)}(x);$  

\medskip

   $ 0 < y_1 < y_2 < \dots < y_n < y_{n+1}$  denote the zeros of $ L_{n+1}^{(\alpha)}(x);$  

\medskip

   $ 0 < w_1 < w_2 < \dots < w_n < w_{n+1}$  denote the zeros of $ L_{n +1}^{(\alpha +2)}(x)$.  

\medskip

From  \cite [Lem.2.1]{DrJo} with $m = 2$ and replacing $n$ by $n +1$, we have

\begin{equation}\label{(2.6)}  x L_{n+1}^{(\alpha +2)}(x) = (x +\alpha +1)  L_{n+1}^{(\alpha +1)}(x)  -  (\alpha +n +2)   L_{n +1}^{(\alpha)}(x),  \end{equation}

Multiplying  \refe{(2.6)}  by $x$  and using the expression in \refe{(2.1)}  for $x L_{n+1}^{(\alpha +1)}(x) $ leads to

\begin{equation}\label{(2.7)}  x^2 L_{n+1}^{(\alpha +2)}(x) = a(x)  L_{n+1}^{(\alpha)}(x)  + b(x)   L_{n}^{(\alpha)}(x),  \end{equation} 

where

\begin{align*}\label{(2.8)} a(x) &=  x^2 - (2n+2)x - (n+1)(\alpha +1); \quad b(x) = (x + \alpha +1)(\alpha +n +1)\\
&=(x-a_1)(x-a_2),  \end{align*} 
with $a_1=  n + 1 + \sqrt{( n +1)(n + \alpha +2)}$, $a_2=  n + 1 - \sqrt{( n +1)(n + \alpha +2)}$.

\medskip

For each $n \in \mathbb{N}$ and each $\alpha > -1,$   the quadratic function $a(x)$ has one positive zero $ a_1$ that tends to $\infty$ as $n\to \infty$ and one negative zero $a_2$ that tends to $- \dfrac{\alpha+1}{2}$ as $n\to\infty$, while the linear factor $b(x)$ has one  negative zero.   Evaluating \refe{(2.7)} at successive positive zeros $w_k$ and $w_{k+1}$ of  $ L_{n +1}^{(\alpha +2)}(x),$ we have for each $k = 1,2,\dots,n$,

\begin{equation}\label{(2.9)}  a(w_k)  a(w_{k+1})  L_{n+1}^{(\alpha)}(w_k) L_{n+1}^{(\alpha)}(w_{k+1})  = b(w_k)  b(w_{k+1})  L_{n}^{(\alpha)}(w_k) L_{n}^{(\alpha)}(w_{k+1}).  \end{equation}

Since $b(x) \neq 0$ for $x$ positive  while  $ L_{n+1}^{(\alpha)}(w_k) L_{n+1}^{(\alpha)}(w_{k+1}) < 0 $ for each $k = 1,2,\dots, n$  because the zeros of  $ L_{n +1}^{(\alpha )}(x)$ and $ L_{n +1}^{(\alpha +2)}(x)$ are interlacing, it follows from \refe{(2.9)} that each of the $n$ intervals $(w_k, w_{k+1})$ has either (i) exactly one zero of  $ L_{n}^{(\alpha)}(x)$ and not the single positive zero of $a(x)$ or (ii) the positive zero of $a(x)$ and no zero of  $ L_{n}^{(\alpha)}(x)$. If (i) holds, the zeros of $ L_{n}^{(\alpha)}(x)$  and $ L_{n +1}^{(\alpha +2)}(x)$ are interlacing while if (ii) holds, $n-1$ zeros of $ L_{n}^{(\alpha)}(x)$ together with the point $ x =  n + 1 + \sqrt{( n +1)(n + \alpha +2)}$ interlace with the zeros of $ L_{n +1}^{(\alpha +2)}(x).$  The same analysis as in Theorem 2.1 shows that if (ii) holds, the smallest zeros of $ L_{n}^{(\alpha)}(x)$, $L_{n +1}^{(\alpha)}(x)$ and  $ L_{n +1}^{(\alpha +2)}(x)$  satisfy

\begin{eqnarray*}   0 <y_1 <  x_1 < w_1 < y_2 <  \dots   \end{eqnarray*}

\medskip

This completes the proof of Theorem \ref{Theorem 2.2.}.
\end{proof}

\begin{table}[!htbp]
\caption{The zeros of $L_{n}^{(\alpha)}(x)$ and $L_{n+1}^{(\alpha+2)}(x)$ for $n=5$ and $\alpha=43$. The point $x= n+1 +\sqrt{(n+1)(n+\alpha +2)}$ is smaller than the smallest root of $L_{n+1}^{(\alpha +2)}(x).$ Full interlacing of the zeros holds.
}
\begin{tabular}{|r|r|r|r|r|r|r|}
\hline
$L_{n+1}^{(\alpha+2)}(x)$ & $w_1$  & $w_2$  & $w_3$  & $w_4$  & $w_5$ & $w_6$  \\
\hline
$L_{n}^{(\alpha)}(x)$   & $x_1$  & $x_2$  & $x_3$  & $x_4$  & $x_5$ & --  \\
\hline\hline & & & & & &  \\
$L_{6}^{(45)}(x)$  & 29.3629 & 37.1301 & 45.0889 & 53.8171 & 63.9142 & 76.6867\\   
$L_{5}^{(43)}(x)$   & 29.6552 & 37.9457 & 46.6608 & 56.6079 & 69.1304 & -- \\
\hline \end{tabular}
\label{Laguerre-5-43}
\end{table}

\begin{table}[!htbp]
\caption{The zeros of $L_{n}^{(\alpha)}(x)$ and $L_{n+1}^{(\alpha+2)}(x)$, for $n=7$ and $\alpha=100$. The point $x= n+1 +\sqrt{(n+1)(n+\alpha +2)}$ is smaller than the smallest root of $L_{n+1}^{(\alpha +2)}(x).$ Full  interlacing of the zeros holds. 
}
\begin{tabular}{|r|r|r|r|r|r|r|r|r|}
\hline
$L_{n+1}^{(\alpha+2)}(x)$ & $w_1$  & $w_2$  & $w_3$  & $w_4$  & $w_5$ & $w_6$ & $w_7$ & $w_8$ \\
\hline
$L_{n}^{(\alpha)}(x)$   & $x_1$  & $x_2$  & $x_3$  & $x_4$  & $x_5$ &  $x_6$ & $x_7$ & --  \\
\hline\hline & & & & & & & & \\
$L_{8}^{(102)}(x)$ & 70.0175 & 81.022 & 91.4896 & 102.139 & 113.376 & 125.623 & 139.552 & 156.781 \\   
$L_{7}^{(100)}(x)$ & 70.9694 & 82.4842 & 93.5606 & 104.995 & 117.321 & 131.252 & 148.418  & -- \\
\hline \end{tabular}
\label{Laguerre-7-100}
\end{table}


\begin{remark} The same technique used in  Theorem 2.1 and Theorem 2.2  can be used to discuss the interlacing of the zeros of $ L_{n}^{(\alpha )}(x)$ and $ L_{n +1}^{(\alpha +3)}(x).$

A calculation using  \refe{(2.6)} with $\alpha$ replaced by $\alpha +1$ together with

\begin{equation*}\label{(2.8a)}  x^2 L_{n+1}^{(\alpha +3)}(x) = (x +\alpha +2) \left[(x+\alpha+1) L_{n+1}^{(\alpha +1)}(x)  -  (\alpha +n +2)   L_{n +1}^{(\alpha)}(x)\right] -(\alpha +n+3)x  L_{n+1}^{(\alpha +1)}(x),  \end{equation*}

and  \refe{(2.1)} leads to

\begin{equation}\label{(2.9a)}  x^2 (x-n-1)L_{n+1}^{(\alpha +3)}(x) = a(x) L_{n+1}^{(\alpha +1)}(x) + b(x) L_{n}^{(\alpha)}(x), \end{equation}

where

\begin{align*}  a(x) &= (x-n-1)[(x+\alpha+2)(x+\alpha+1)-x(\alpha+n+3)]-x(x+\alpha+2)(\alpha+n+2),\\
b(x) &= (\alpha+n+1)(\alpha+n+2)(x +\alpha+2).
\end{align*}

Evaluating \refe{(2.9a)} at successive positive zeros $x_k$ and $x_{k+1}$ of $L_{n+1}^{(\alpha +3)}(x),$  we know that $L_{n+1}^{(\alpha +1)}(x_k)L_{n+1}^{(\alpha +1)}(x_{k+1}) <0$ since the zeros of two equal degree Laguerre polynomials with parameter difference $2$  are interlacing. Also $b(x_k) b(x_{k+1}) >0$ for each $k \in {1,2,\dots,n}$ since $x +\alpha +2>0$ for $x>0,\alpha>-1.$  It follows that the location (positivity and real/complex) of the zeros of the cubic function $a(x)$ will determine the extent of the interlacing of the zeros of $ L_{n}^{(\alpha )}(x)$ and $ L_{n +1}^{(\alpha +3)}(x).$
\end{remark}

\section{Sharp $t-$ interval for interlacing of zeros of $L_n^{(\alpha)}(x)$ and $L_n^{(\alpha +t)}(x),$  $\alpha > -1, t >0$ }

Using results proved by P\'almai, \cite{Pal}, for the sharp interlacing of  the positive zeros of Bessel functions $J_{\mu}(x)$ and $J_{\nu}(x)$, it was proved in  \cite{DrMu2} that  $0 < t \leq 2$ is the sharp $t-$ interval within which the $n$ zeros of $L_n^{(\alpha)}(x)$ and $L_n^{(\alpha +t)}(x),$  $t >0,$ are  interlacing for every value of $n \in \mathbb{N}$ and every value of $\alpha >-1$.   More precisely, given any $\epsilon >0,$ there exists $n \in \mathbb{N},$ $\alpha > -1,$  such that the zeros of  $L_n^{(\alpha)}(x)$ and $L_n^{(\alpha +2 +\epsilon)}(x)$ are not interlacing. Full interlacing of the zeros breaks down when  $n \in \mathbb{N}$ is large. The question arises as to whether interlacing of the zeros of $L_n^{(\alpha)}(x)$ and $L_n^{(\alpha +t)}(x)$  breaks down in a structured way in the sense that there is partial, but not full,  interlacing  of zeros  of  $L_n^{(\alpha)}(x)$ and $L_n^{(\alpha +t)}(x),$  $\alpha > -1 $ when $t >2.$  

\bigskip

\begin{theorem}\label{Theorem 3.1.} Fix $ \alpha > -1$ and suppose $\{L_{n}^{(\alpha)}(x)\} _{n=0}^\infty $ is  the sequence of Laguerre polynomials. Assume that $\alpha$ and $n \in \mathbb{N}$ are such that $ L_{n}^{(\alpha +3)}(x)$ and $ L_{n}^{(\alpha)}(x)$ have no common zeros.   The zeros of  $ L_{n}^{(\alpha +3)}(x)$  and $L_n^{(\alpha)}(x)$ are interlacing if  $\alpha +1 \geq n.$  For   $\alpha +1 < n$ and $n \geq 3,$   at least $n-2$  zeros of  $ L_{n}^{(\alpha )}(x)$   interlace with the zeros of $L_n^{(\alpha +3)}(x).$ 
\end{theorem}

\begin{proof} From  \cite [Lem.2.1]{DrJo} with $m = 3$ and $m = 2$ respectively, we have

\begin{equation}\label{(3.1)}  x L_{n}^{(\alpha +3)}(x) = (x +\alpha +2)  L_{n}^{(\alpha +2)}(x)  -  (\alpha +n +2)   L_{n}^{(\alpha +1)}(x) ,  \end{equation}

and

\begin{equation}\label{(3.2)}  x L_{n}^{(\alpha +2)}(x) = (x +\alpha +1)  L_{n}^{(\alpha +1)}(x)  -  (\alpha +n +1)   L_{n}^{(\alpha)}(x).  \end{equation} 

Multiplying \refe{(3.1)} by $x$ and using \refe{(3.2)}, 

\begin{equation}\label{(3.3)}  x^2 L_{n}^{(\alpha +3)}(x) = a(x)  L_{n}^{(\alpha +1)}(x)  -   b( x)   L_{n}^{(\alpha)}(x)    \end{equation} 
 
where

\begin{align}\label{(3.4)}   a(x) =  x^2 + (\alpha + 1 - n) x +(\alpha +1)(\alpha +2) ; \quad  b(x) =  (\alpha +  n +1 )(x +\alpha +2) .\end{align}

\medskip

The quadratic equation $a(x) =0$  has $2$ real zeros or a single zero of order $2$ or $2$ complex conjugate zeros depending on the values of $n \in \mathbb{N} $ and $\alpha > -1$  while the linear equation $b(x) =0$ has a single negative zero at $x = -\alpha -2$.  For every $\alpha > -1$, $a(x)$ is a monic quadratic function with positive constant term $(\alpha+1)(\alpha +2).$  It follows  from \refe{(3.4)} that  $a(x)$ has two positive zeros when $\alpha +1 <n;$   two negative zeros when $\alpha +1 > n;$  and  two pure imaginary zeros when  $\alpha +1 =  n.$  In other words, $a(x)$ has no positive zeros when $\alpha +1 \geq n.$

 Evaluating \refe{(3.3)} at successive positive zeros $w_k$ and $w_{k+1},$  $k = 1,2,\dots,n-1,$  of  $ L_{n}^{(\alpha +3)}(x),$   we have

\begin{equation}\label{(3.5)}  a(w_k)  a(w_{k+1})  L_{n}^{(\alpha +1)}(w_k) L_{n}^{(\alpha +1)}(w_{k+1})  =  b(w_k)  b(w_{k+1})  L_{n}^{(\alpha)}(w_k) L_{n}^{(\alpha)}(w_{k+1}).  \end{equation} 

Since  the zeros of  $ L_{n}^{(\alpha +1 )}(x)$ and $ L_{n}^{(\alpha +3)}(x)$ are interlacing  (equal degree with parameter  difference  $2$),  we know that  $ L_{n}^{(\alpha +1)}(w_k) L_{n}^{(\alpha +1)}(w_{k+1}) < 0 $ for each $k = 1,2,\dots, n -1.$  Also,  $b(w_k)  b(w_{k+1})>0$ for each $k = 1,2,\dots, n -1,$ since $b(x)$ has no positive zeros.  If $\alpha +1 \geq n,$ it follows from \refe{(3.5)} that  $ L_{n}^{(\alpha)}(x)$  has a different sign at successive positive zeros $w_k$ and $w_{k+1},$  $k = 1,2,\dots,n-1,$  of  $ L_{n}^{(\alpha +3)}(x).$ Also, from \cite[Th. 2.3]{DrJo},  the smallest zero $x_1$ of $ L_{n}^{(\alpha)}(x)$  satisfies  $ 0 <x_1 <  w_1$ and  it follows that the zeros of $ L_{n}^{(\alpha)}(x)$  and $ L_{n}^{(\alpha +3)}(x)$ are interlacing if $\alpha +1 \geq n$.   If  $\alpha +1  < n, $  $ a(x)$ has $2$  positive zeros and at least $n-2$ zeros of $ L_{n}^{(\alpha)}(x)$ interlace with the  zeros $w_k$ and $w_{k+1}$  of  $ L_{n}^{(\alpha +3)}(x).$   The precise number of zeros of  $ L_{n}^{(\alpha)}(x)$ and $ L_{n}^{(\alpha +3)}(x)$ that are interlacing when $ \alpha +1 < n$ depends on the location of the two zeros of $a(x)$ relative to the $n$ zeros of $ L_{n}^{(\alpha +3)}(x)$.

\end{proof}

\begin{remark}
The numerical results in Table \ref{Laguerre-2} confirm that a definitive statement cannot be made about the interlacing of the zeros of $L_n^{(\alpha)}(x)$ and $L_n^{(\alpha +3)}(x)$ when $\alpha +1 <n.$ Suppose that  $(i)$ $\alpha = 0$, $n=7$. Then $\alpha +1 <7$ and Table \ref{Laguerre-2} shows that there is no zero of $ L_{7}^{(3)}(x)$ between the two smallest zeros $x_1$ and $x_2$ of  $ L_{7}^{(0)}(x)$.  Further, when $\alpha = 0$, $n=7$,  we see from \refe{(3.4)} that  $a(x) = x^2 - 6 x +2$  which has distinct positive zeros at $x = 3 - \sqrt{7}$ and $x = 3 + \sqrt{7}$. The smaller of these two zeros of $a(x)$ "fills the gap" between the two smallest  zeros $x_1$ and $x_2$ of  $ L_{7}^{(0)}(x)$ that does not contain a zero of $ L_{7}^{(3)}(x)$ while the larger zero of $a(x)$ lies in the "boxed interval" that contains the point $x = n +1 = 8$. Therefore, for these values of $\alpha$ and $n$ with $\alpha +1 <n$, the zeros of  $L_n^{(\alpha)}(x)$ and $L_n^{(\alpha +3)}(x)$ are not interlacing.  Now suppose  $(ii)$ that $\alpha = 1$, $n=7$ so that  $\alpha +1 <7$. Table \ref{Laguerre-2} shows that the zeros of $ L_{7}^{(1)}(x)$ and  $ L_{7}^{(4)}(x)$  are interlacing. Here, when $\alpha = 1$, $n=7$,  the two distinct positive zeros of  $a(x) = x^2 - 5 x +6$ are at $x =2$ and $x = 3$ which  lie in the same interval with endpoints at the successive zeros $1.55$ and $3.34$ of $ L_{7}^{(4)}(x).$  We can see that Table \ref{Laguerre-2} is consistent with Theorem 3.1 when  $\alpha +1 =n$ and when $\alpha +1 >n.$ For example,  if $\alpha =6$ and $n=7$ so that  $\alpha +1 =n$, the zeros of $ L_{7}^{(6)}(x)$ and the zeros of  $ L_{7}^{(9)}(x)$  are interlacing as expected. Also, if  $\alpha = 11$ and $n=7,$ so that $\alpha +1 >n,$ the zeros of $ L_{7}^{(11)}(x)$ and  $ L_{7}^{(14)}(x)$  are interlacing.
\end{remark}

\medskip

We now consider  interlacing of the zeros of equal degree Laguerre polynomials  $ L_{n}^{(\alpha +4)}(x)$  and $L_n^{(\alpha)}(x)$.

\medskip

\begin{theorem}\label{Theorem 3.2.}  Fix $ \alpha > -1$ and suppose $\{L_{n}^{(\alpha)}(x)\} _{n=0}^\infty $ is  a sequence of Laguerre polynomials.   Let  $n \in \mathbb{N}, n \geq 3.$ Assume that $\alpha$ and $n \in \mathbb{N}$ are such that $ L_{n}^{(\alpha +4)}(x)$ and $ L_{n}^{(\alpha)}(x)$ have no common zeros.  The zeros of   $L_{n}^{(\alpha +4)}(x)$  and  $ L_{n}^{(\alpha)}(x)$  are interlacing for values of $n$ and $\alpha$ for which the following conditions hold: The cubic polynomial  $p_3(x) = p_3(x,n,\alpha)$  defined by \refe{(3.11)} below has one real zero $l_1 = l_1(\alpha,n)$ which is $ < w_1,$ where $w_1$ is the smallest (positive) zero of  $ L_{n}^{(\alpha +4)}(x)$  and the remaining two zeros of $p_3(x)$ are either negative, or have a  non-zero imaginary part. 
\end{theorem}

\begin{proof}   From  \cite [Lem.2.1]{DrJo} with $m = 4$ and $m = 3$ respectively, we have

\begin{equation}\label{(3.7)}  x L_{n}^{(\alpha +4)}(x) = (x +\alpha +3)  L_{n}^{(\alpha +3)}(x)  -  (\alpha +n +3)   L_{n}^{(\alpha +2)}(x) ,  \end{equation}

\begin{equation}\label{(3.8)}  x L_{n}^{(\alpha +3)}(x) = (x +\alpha +2)  L_{n}^{(\alpha +2)}(x)  -  (\alpha +n +2)   L_{n}^{(\alpha +1)}(x).  \end{equation} 

Multiplying \refe{(3.7)} by $x$ and using \refe{(3.8)}, we have

\begin{equation}\label{(3.9)}  x^2 L_{n}^{(\alpha +4)}(x) = d(x)  L_{n}^{(\alpha +2)}(x)  -   e( x)   L_{n}^{(\alpha +1)}(x),    \end{equation} 
 
where  $d(x) =  (x + \alpha + 3)( x + \alpha +2)- x(\alpha +n +3) $  and $ e(x) =  (x + \alpha + 3)  (\alpha +  n +2 )$.

\medskip

Also,  from  \cite [Lem.2.1]{DrJo} with $m = 2$

\begin{equation}\label{(3.10)}  (x + \alpha + 1) L_{n}^{(\alpha +1)}(x) =  x  L_{n}^{(\alpha +2)}(x)  +  (\alpha +n +1)   L_{n}^{(\alpha)}(x),    \end{equation} 

so that multiplying \refe{(3.9)} by $(x +\alpha +1) $ and using   \refe{(3.10)}, we have

\begin{equation}\label{(3.11)}   x^2  (x + \alpha + 1) L_{n}^{(\alpha +4)}(x) =  p_{3}(x)   L_{n}^{(\alpha +2)}(x)  -  b_{1}(x)   L_{n}^{(\alpha)}(x),   \end{equation}

where  

\begin{equation*}\label{(3.12)}  p_{3}(x) = (x + \alpha + 3)(x + \alpha + 2)(x + \alpha + 1) - x  ( x + \alpha + 1)  (\alpha +n +3 ) -  x(x + \alpha +3) (\alpha +n +2), \end{equation*}

is a polynomial of degree $3$ in $x$ and 

\begin{equation*} b_{1}(x) = ( x + \alpha + 3)  (\alpha +n +2 )(\alpha +n +1 ), \end{equation*} 

is linear in $x$ with one negative zero. Moreover,

$$
p_3(x)=(x-l_1(\alpha,n))(x-l_2(\alpha,n))(x-l_3(\alpha,n)),
$$
where
\begin{align*}
l_1(\alpha,n)&=\frac{1}{3} \left(2n-1 - \alpha  -\frac{h_3}{\left(h_2 + \sqrt{h_1}\right)^{1/3}} + \left(h_2 + \sqrt{h_1}\right)^{1/3}\right),\\
l_2(\alpha,n)&=\frac{1}{3} \left(2n-1 - \alpha + \frac{(1 +  i \sqrt{3})h_3}{2\left(h_2 + 
       \sqrt{h_1}\right)^{1/3}} 
   -\dfrac{2\left(h_2 + \sqrt{h_1}\right)^{1/3}}{(1 + i\sqrt{3})}\right),\\
l_3(\alpha,n)&=\frac{1}{3} \left(2n-1 - \alpha + \frac{(1 -  i \sqrt{3}) h_3}{2\left(h_2 +  \sqrt{h_1}\right)^{1/3}} 
    -\dfrac{2\left(h_2 + \sqrt{h_1}\right)^{1/3}}{(1 - i\sqrt{3})}\right),
\end{align*}
and   
\begin{align*}
h_{1} = h_{1}(\alpha,n)&=(\alpha-2 n+1)^3 (2 \alpha+2 n+5)^3\\
&+\left(10 \alpha^3+66 \alpha^2+129 \alpha-8 n^3-6 (\alpha+4) n^2+6 (\alpha+1) (2 \alpha+5) n+73\right)^2,\\
h_{2} = h_{2}(\alpha,n)&=-10 \alpha^3-6 (2 n+11) \alpha^2+3 (2 (n-7) n-43) \alpha+8 n^3+24 n^2-30 n-73,\\
h_{3} = h_{3}(\alpha,n)&=(\alpha-2 n+1) (2 \alpha+2 n+5).
\end{align*}

Evaluating  \refe{(3.11)} at successive positive zeros $w_k$ and $w_{k+1}$  of $ L_{n}^{(\alpha +4)}(x),$ $k = 1,2,\dots,n-1,$ we have

\begin{equation*}  p_{3} (w_k)  p_{3}({w_{k+1}})   L_{n}^{(\alpha +2)}(w_k) L_{n}^{(\alpha +2)}(w_{k+1})  =  b_1(w_k)  b_1(w_{k+1})  L_{n}^{(\alpha)}(w_k) L_{n}^{(\alpha)}(w_{k+1}).  \end{equation*}

We know that  the zeros of  $ L_{n}^{(\alpha +2 )}(x)$ and $ L_{n}^{(\alpha +4)}(x)$ are interlacing  (equal degree with parameter difference  $2$) so  $ L_{n}^{(\alpha +2)}(w_k) L_{n}^{(\alpha +2)}(w_{k+1}) < 0 $ for each $k = 1,2,\dots, n -1$.  Also,  $b_1(w_k) b_1(w_{k+1})>0,$ since $b_1(x)$ has one negative zero at $x = -\alpha -3$ and does not change sign for any positive value of $x$.   Assuming that $l_2(\alpha,n)$ and  $l_3(\alpha,n)$ are complex with non-vanishing imaginary part for $\alpha >-1,$ $n \in \mathbb{N},$ it follows that $p_3(x)$ has at most one positive simple zero $l_1$ greater than $w_1.$ If  $l_1 <w_1, $ the zeros of  $ L_{n}^{(\alpha +4)}(x)$ and $ L_{n}^{(\alpha)}(x)$ are interlacing while if  $l_1 >w_1, $  each interval $(w_k, w_{k+1}),$  $k = 1,2,\dots,n-1,$ with endpoints at successive zeros of  $ L_{n}^{(\alpha +4)}(x)$ contains either exactly one zero of $ L_{n}^{(\alpha)}(x)$  and not the point $l_1$ or the point $l_1$ and no zero of $ L_{n}^{(\alpha)}(x).$

\end{proof}

\begin{remark} The zeros of  $ L_{n}^{(\alpha +4 )}(x)$ and $ L_{n}^{(\alpha)}(x)$ interlace only if  $n <<\alpha,$ i.e., for a small and restricted set of polynomials. \end{remark}

\begin{center}
\includegraphics[scale=0.99]{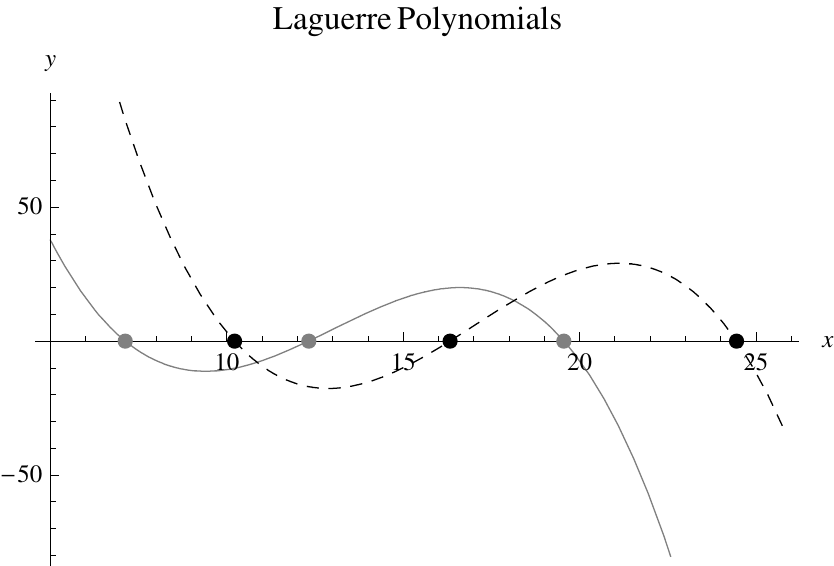}
\captionof{figure}{The roots of $L_{3}^{10}(x)$ are depicted by dots in gray and those of $L_{3}^{14}(x)$ are the black dots.}
\label{fig_Laguerre_zeros_3-3_2-6}
\end{center}

\section{Sharp $t-$ interval for interlacing of zeros of $L_n^{(\alpha)}(x)$ and $L_{n-1}^{(\alpha +t)}(x),$  $\alpha > -1, t >0$ }

It is known \cite [Th. 3.1]{DrMu1} that  the sharp $t-$ interval within which the zeros of the Laguerre polynomials $L_{n-1}^{(\alpha +t)}(x)$ and the  zeros of  $L_n^{(\alpha)}(x), \alpha > -1$ are interlacing is $0 \leq t \leq 2.$  The question arises as to what extent (if any)  partial interlacing holds between the zeros of $L_{n-1}^{(\alpha +t)}(x)$ and  $L_n^{(\alpha)}(x),$  $\alpha > -1,$ when $t >2$. We consider the cases $t=3$ and $t=4$ and use two mixed three term recurrence relations \cite [(3.18)]{DrMu1} and \cite [(3.13)]{DrMu1}, namely

\begin{equation}\label{(4.1)}  x^2  L_{n-1}^{(\alpha +2)}(x) = -n  (x +\alpha +1)  L_{n}^{(\alpha)}(x)  +(\alpha +1)   (\alpha +n)   L_{n-1}^{(\alpha)}(x),  \end{equation}

and

\begin{equation} \label{(4.2)}  L_{n}^{(\alpha +1)}(x) =   L_{n}^{(\alpha)}(x)  +    L_{n-1}^{(\alpha +1)}(x).  \end{equation} 

\bigskip

\begin{theorem}\label{Theorem 4.1.} Suppose $\{L_{n}^{(\alpha)}(x)\} _{n=0}^\infty $ is  a sequence of Laguerre polynomials with $ \alpha > -1$ and let $\{x_{i}\} _{i=1}^{n} $ be the zeros of  $ L_{n}^{(\alpha)}(x)$  in increasing order. Assume that $\alpha$ and $n \in \mathbb{N}$ are such that $ L_{n-1}^{(\alpha +3)}(x)$ and $ L_{n}^{(\alpha)}(x)$ have no common zeros.   Let $ k_n = k_n(\alpha) = \dfrac{(\alpha +1)(\alpha +2)}{n}.$  For each $n \in \mathbb{N},$  each interval $ (0,x_1), (x_1, x_{2}), \dots, (x_{n-1}, x_n) $ contains either a simple zero of  $ L_{n-1}^{(\alpha +3)}(x)$ or  the point $k_{n},$  but not both. 
\end{theorem}

\begin{proof}  From  \refe{(4.1)} with $\alpha$ replaced by $\alpha +1,$  we have

\begin{equation}\label{(4.3)}  x^2  L_{n-1}^{(\alpha +3)}(x) = -n  (x +\alpha +2)  L_{n}^{(\alpha +1)}(x)  +(\alpha +2)   (\alpha +n +1)   L_{n-1}^{(\alpha +1)}(x),  \end{equation}

and using \refe{(4.2)} we have, after simplification,  

\begin{equation}\label{(4.4)}   x^2 L_{n-1}^{(\alpha +3)}(x) = -n  (x +\alpha +2)  L_{n}^{(\alpha)}(x)  +  [(\alpha +2)   (\alpha +1) - nx]   L_{n-1}^{(\alpha +1)}(x).  \end{equation} 

We know from \cite[Th. 2.4]{DrJo} that the zeros of  $L_{n}^{(\alpha)}(x)$ and  $ L_{n-1}^{(\alpha +1)}(x)$ are interlacing. Evaluating \refe{(4.4)} at successive positive zeros of $L_{n}^{(\alpha)}(x),$ we have

\begin{equation*}\label{(4.5)}   {x_k}^2  {x_{k+1}}^2  L_{n-1}^{(\alpha +3)}(x_k)  L_{n-1}^{(\alpha +3)}(x_{k+1}) = a(x_{k}) a( x_{k+1})L_{n-1}^{(\alpha +1)}(x_k)  L_{n-1}^{(\alpha +1)}(x_{k+1}), \end{equation*}  

where

\begin{equation*}\label{(4.5)} a(x) = (\alpha +1 )  (\alpha + 2) - nx,    \end{equation*} 

and the result follows using the same analysis as in Theorem 2.1.
\end{proof}

\begin{center}
\includegraphics[scale=0.95]{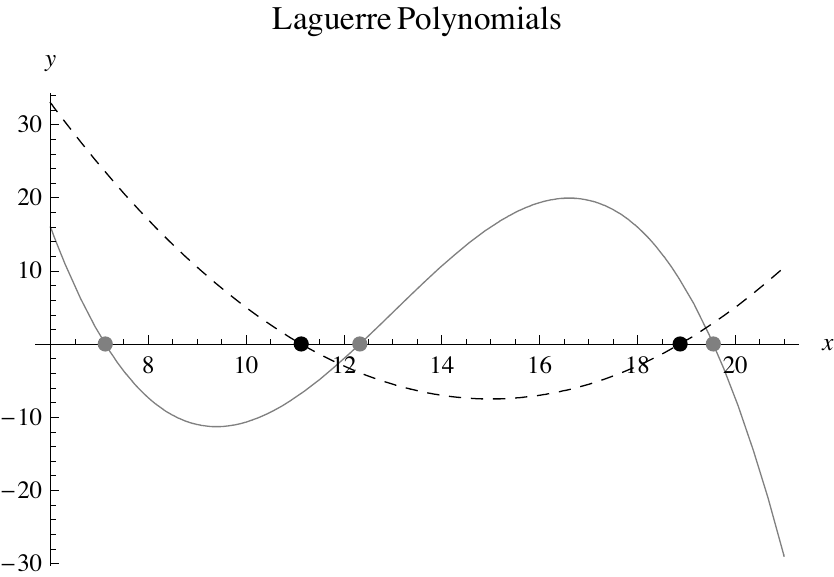}
\captionof{figure}{The zeros of $L_{3}^{10}(x)$ are depicted by gray dots  and those of $L_{2}^{13}(x)$ by black dots. Note that interlacing holds in this example because $\alpha =10$ and $n =3$ so that $ (\alpha +1 )  (\alpha + 2)/ n  = 44,$ which is larger than the largest zero of  $L_n^{(\alpha)}(x).$}
\label{fig_Laguerre_zeros_3-2_10-13} 
\end{center}

\begin{theorem}\label{Theorem 4.2.} Suppose $\{L_{n}^{(\alpha)}(x)\} _{n=0}^\infty $ is  a sequence of Laguerre polynomials with $ \alpha > -1$ and let $\{x_{i}\} _{i=1}^{n} $ be the zeros of  $ L_{n}^{(\alpha)}(x)$  in increasing order.  Assume that $\alpha$ and $n \in \mathbb{N}$ are such that $ L_{n-1}^{(\alpha +4)}(x)$ and $ L_{n}^{(\alpha)}(x)$ have no common zeros. Let $q_{+} = q_{+}{(n,\alpha)}$ be the positive root of the quadratic equation  $n x^2 + 2n (\alpha +2) x - (\alpha +1)(\alpha +2) (\alpha +3) =0$.   For each $n \in \mathbb{N},$  $n \geq 4,$  at least $n-2$ of the $n-1$ intervals $(x_1,x_2), (x_2, x_3) \dots           (x_{n-1},x_n)$ with endpoints at successive zeros of  $L_n^{(\alpha)}(x)$ contain one zero of $ L_{n-1}^{(\alpha +4)}(x).$  If $q_{+} < x_1$ or $q_{+} > x_n,$  the zeros of $ L_{n-1}^{(\alpha +4)}(x)$  interlace with the $n$ zeros of $L_n^{(\alpha)}(x)$.  
\end{theorem}

\begin{proof} From  \refe{(4.3)} with $\alpha$ replaced by $\alpha +1,$  we have

\begin{equation}\label{(4.5)}  x^2  L_{n-1}^{(\alpha +4)}(x) = -n  (x +\alpha +3)  L_{n}^{(\alpha +2)}(x)  +(\alpha +3)   (\alpha +n +2)   L_{n-1}^{(\alpha +2)}(x).  \end{equation}

Also, from \refe{(4.1)} with $\alpha$ replaced by $\alpha +1,$ 

\begin{equation*}\label{(4.6)}   L_{n-1}^{(\alpha +2)}(x) =   L_{n}^{(\alpha +2)}(x)  -    L_{n}^{(\alpha +1)}(x),  \end{equation*} 

and substituting  into \refe{(4.5)} we have  

\begin{equation*}  x^2 L_{n-1}^{(\alpha +4)}(x) = - n  (x +\alpha +3)  L_{n}^{(\alpha +2)}(x)  +  (\alpha +3)   (\alpha +n +2) \left( L_{n}^{(\alpha +2)}(x) -  L_{n}^{(\alpha +1)}(x)\right),   \end{equation*} 
 
which simplifies to

\begin{equation}\label{(4.8)}   x^2 L_{n-1}^{(\alpha +4)}(x) = a (x)  L_{n}^{(\alpha+2)}(x)  +  b_{n,\alpha}  L_{n}^{(\alpha +1)}(x),  \end{equation} 

where 

\begin{equation*}\label{(4.9)} a(x) = (\alpha +3) (\alpha + n +2) - n(x + \alpha +3) , \, \, \,  b_{n,\alpha} = -  (\alpha +3) (\alpha + n +2). \end{equation*}

Multiplying \refe{(4.8)} by $(x + \alpha +1)$ and using \cite[(3.12)]{DrMu1}, gives

\begin{equation*}\label{(4.10)}  x^2 (x + \alpha +1)  L_{n-1}^{(\alpha +4)}(x) =[ (x +\alpha +1) a(x) + x b_{n,\alpha}]   L_{n}^{(\alpha +2)}(x)  +  b_{n,\alpha} (\alpha +n +1)   L_{n}^{(\alpha)}(x),  \end{equation*}

or

\begin{equation*}\label{(4.11)}  x^2 (x + \alpha +1)  L_{n-1}^{(\alpha +4)}(x) = q(x)  L_{n}^{(\alpha +2)}(x)  +  b_{n,\alpha} (\alpha +n +1)   L_{n}^{(\alpha)}(x).  \end{equation*}

A calculation shows that the coefficient of $ L_{n}^{(\alpha +2)}(x) $  is a quadratic function $q(x) = -n r(x),$ where $r(x)=(x-q_{-})(x-q_{+})$, where  $q_{-}<0$ and $q_{+}>0$.

Since  the zeros of  $ L_{n}^{(\alpha )}(x)$ and $ L_{n}^{(\alpha +2)}(x)$ are  positive and  interlacing  (equal degree, parameter difference $2$), we know that  $ L_{n}^{(\alpha +2)}(x_k) L_{n}^{(\alpha +2)}(x_{k+1}) < 0 $  for each $ k \in {\{ 1,2,\dots, n-1}\},$ while $q(x_k)  q(x_{k+1})>0$  except for at most $1$ value of $k$, $ k \in {\{ 1,2,\dots, n-1}\}$.   Therefore, $ L_{n-1}^{(\alpha +4)}(x)$  has a different sign at successive positive zeros $x_k$ and $x_{k+1}$  of  $ L_{n}^{(\alpha)}(x)$  for at least $n-2$ values of  $k \in {\{ 1,2,\dots,n-1}\}$.  Therefore, interlacing of the zeros $ L_{n-1}^{(\alpha +4)}(x)$ with the zeros of $ L_{n}^{(\alpha)}(x)$ fails in a maximum of $1$ interval. The interval in which interlacing fails (if it does) depends on the value of the real positive zero $x_{+}$ of $q(x)$ relative to the largest and smallest zeros of $L_n^{\alpha} (x).$

\end{proof}

\begin{center}
\includegraphics[scale=0.95]{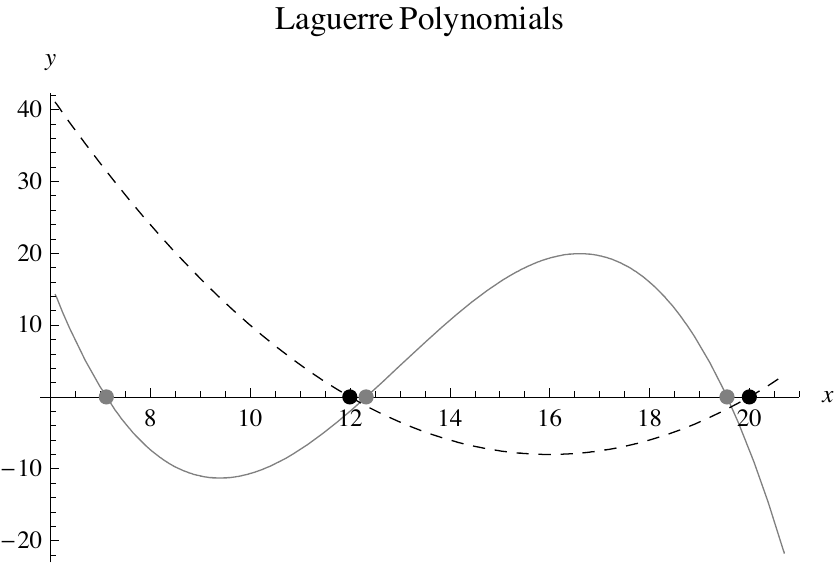}
\captionof{figure}{The roots of $L_{3}^{10}(x)$ are depicted by dots in gray and those of $L_{2}^{14}(x)$ are the black dots. We see that the zeros are not interlacing.
}
\label{fig_Laguerre_zeros_3-2_10-14}
\end{center}

\begin{table}[!htbp]
\caption{The zeros of $L_{n}^{(\alpha)}(x)$ and $L_{n-1}^{(\alpha+4)}(x)$, for $n=7$ and $\alpha=-1/2$. The 
positive zero $q_{+}$ of $q(x)$ lies between the smallest and the largest zeros of $L_{n}^{(\alpha)}(x).$ The interval with endpoints at successive zeros of $L_{n}^{(\alpha)}(x)$ where the interlacing of zeros breaks down is highlighted with boxes.

}
\begin{tabular}{|r|r|r|r|r|r|r|r|}
\hline
$L_{n}^{(\alpha)}(x)$ & $x_1$  & $x_2$  & $x_3$  & $x_4$  & $x_5$ & $x_6$  & $x_7$\\
\hline
$L_{n-1}^{(\alpha+4)}(x)$   & $X_1$  & $X_2$  & $X_3$  & $X_4$  & $X_5$ & $X_6$ & --  \\
\hline\hline & & & & & &  &\\
$L_{7}^{(-1/2)}(x)$ & $\fbox{0.0851}$ & $\fbox{0.7721}$ & 2.1806 & 4.3898 & 7.5541 & 11.9900 & 18.5283  \\   
$L_{6}^{(7/2)}(x)$ & 1.5135 & 3.4321 & 6.1108 & 9.7233 & 14.6039 & 21.6165  & -- \\
\hline \end{tabular}
\label{Laguerre-7--1/2}
\end{table}

\begin{table}[!htbp]
\caption{The zeros of $L_{n}^{(\alpha)}(x)$ and $L_{n-1}^{(\alpha+4)}(x)$ are interlacing when $n=6$ and $\alpha=140.$ The 
positive zero $q_{+}$ of $q(x)$ is larger than the largest zero of $L_{n}^{(\alpha)}(x).$ 
}
\begin{tabular}{|r|r|r|r|r|r|r|}
\hline
$L_{n}^{(\alpha)}(x)$ & $x_1$  & $x_2$  & $x_3$  & $x_4$  & $x_5$ & $x_6$  \\
\hline
$L_{n-1}^{(\alpha+4)}(x)$   & $X_1$  & $X_2$  & $X_3$  & $X_4$  & $X_5$  & --  \\
\hline\hline & & & & & &  \\
$L_{6}^{(140)}(x)$ & 107.898 & 122.754 & 137.03 & 151.886 & 168.291 & 188.141 \\   
$L_{5}^{(144)}(x)$ & 115.547 & 131.765 & 147.665 & 164.792 & 185.231 & -- \\
\hline \end{tabular}
\label{Laguerre-6-140}
\end{table}

\begin{table}[!htbp]
\caption{The zeros of $L_{n}^{(\alpha)}(x)$ and $L_{n-1}^{(\alpha+4)}(x)$, for $n=8$ and $\alpha=50.$ The 
positive zero $q_{+}$ of $q(x)$ lies between the smallest and the largest zeros of $L_{n}^{(\alpha)}(x).$  The interval with endpoints at successive zeros of $L_{n}^{(\alpha)}(x)$ where the interlacing of zeros breaks down is highlighted with boxes.
}
\begin{tabular}{|r|r|r|r|r|r|r|r|r|}
\hline
$L_{n}^{(\alpha)}(x)$ &  $x_1$  & $x_2$  & $x_3$  & $x_4$  & $x_5$  & $x_6$ & $x_7$ & $x_8$ \\
\hline
$L_{n-1}^{(\alpha+4)}(x)$   & $X_1$  & $X_2$  & $X_3$  & $X_4$  & $X_5$ & $X_6$ & $X_7$ & --  \\
\hline\hline & & & & & & & & \\
$L_{8}^{(50)}(x)$ & 29.962 & 37.123 & 44.225 & 51.695 & 59.809 & 68.883 & $\fbox{79.455}$ & $\fbox{92.849}$  \\   
$L_{7}^{(54)}(x)$ & 34.500 & 42.514 & 50.503 &
   58.991 &  68.373 & 79.223 & 92.896  & -- \\
\hline \end{tabular}
\label{Laguerre-8-50}
\end{table}

\section*{Acknowledgement} Jorge Arves\'u and Kathy Driver wish to thank the Mathematics Department at Baylor University for hosting their visits in Fall 2019 which stimulated this research.

\end{document}